\documentclass[a4paper,12pt]{article}

\usepackage{color,makeidx,bm,latexsym,amscd,amsmath,amssymb,stmaryrd,amsthm,float,graphicx,psfrag}

\theoremstyle{plain}
\newtheorem{theorem}{Theorem}[section]
\newtheorem{corollary}[theorem]{Corollary}
\newtheorem{lemma}[theorem]{Lemma}

\theoremstyle{definition}
\newtheorem{definition}[theorem]{Definition}
\newtheorem{remark}[theorem]{Remark}
\newtheorem{example}[theorem]{Example}
\newtheorem{proposition}[theorem]{Proposition}

\newcommand{\bN}{\ensuremath{\mathbb{N}}}
\newcommand{\bZ}{\ensuremath{\mathbb{Z}}}

\newcommand{\Hom}{\operatorname{Hom}}
\newcommand{\Set}{\operatorname{{\bf Set}}}
\newcommand{\Graph}{\operatorname{\bf Graph}}
\newcommand{\bij}{\operatorname{bij}}
\newcommand{\inj}{\operatorname{inj}}
\newcommand{\surj}{\operatorname{surj}}
\newcommand{\St}{\operatorname{{\bf St}}}

\newcommand{\kasen}{\underline{\chi}}
\newcommand{\maru}{\bullet}
\newcommand{\id}{\operatorname{id}}

\title{Chromatic functors of graphs}
\author{Masahiko Yoshinaga\thanks{Department of Mathematics, 
Hokkaido University, North 10, West 8, Kita-ku, 
Sapporo 060-0810, JAPAN 
E-mail: yoshinaga@math.sci.hokudai.ac.jp}}
\date{\today}
\begin{document}
\maketitle

\begin{abstract}
Finite graphs that have a common chromatic polynomial have the same number of regular $n$-colorings. 
A natural question is whether there exists a natural bijection between regular $n$-colorings. We address this question using a functorial formulation. Let $G$ be a simple graph. Then for each set $X$ we can associate a set of $X$-colorings. This defines a functor, ``chromatic functor'' from the category of sets with injections to itself. The first main result verifies that two finite graphs determine isomorphic chromatic functors if and only if they have the same chromatic polynomial.

Chromatic functors can be defined for arbitrary, possibly infinite, graphs. This fact enables us to investigate functorial chromatic theory for infinite graphs. We prove that chromatic functors satisfy the Cantor-Bernstein-Schr\"oder property. We also prove that countable connected trees determine isomorphic chromatic functors. Finally, we present a pair of infinite graphs that determine non-isomorphic chromatic functors.
%
\end{abstract}

\section{Introduction}
\label{sec:intro}

A simple graph $G=(V, E)$ consists of a set of vertices $V$ and a set of unoriented edges $E\subset 2^V$ that connect two vertices. A regular $n$-coloring of $G$ is a map $c:V\longrightarrow [n]=\{1, 2, \dots, n\}$ such that $\{v, v'\}\in E$ implies $c(v)\neq c(v')$. 

It is known that, for a finite graph $G=(V, E)$, there exists a polynomial $\chi(G, t)\in\bZ[t]$ such that the number of regular $n$-colorings equals $\chi(G, n)$ (\cite{bir, read}), where $\chi(G, t)$ is called the chromatic polynomial of $G$. The minimum positive integer $n$ satisfying $\chi(G, n)\neq 0$ is called the chromatic number of $G$ and denoted by $\chi(G)$. 

Two finite graphs $G_1$ and $G_2$ are said to be chromatically equivalent if $\chi(G_1, t)=\chi(G_2, t)$ (\cite{bir}). Non-isomorphic graphs may be chromatically equivalent, for example, if $G=(V, E)$ is a finite connected tree, then $\chi(G, t)=t(t-1)^{\#V-1}$. Thus, connected trees that have the prescribed $\#V$ are chromatically equivalent. Therefore, it is a natural question to ask whether there exists a certain bijection between the sets of regular $n$-colorings of chromatically equivalent graphs.

To address the above question, we introduce a functorial formulation. 
Let $X$ be a set. We consider $X$ to be a set of colors, and define the set of regular $X$-coloring by 
\begin{equation}
\label{eq:chfct}
\underline{\chi}(G, X):=
\{c:V\longrightarrow X\mid 
\{v_1, v_2\}\in E\Longrightarrow c(v_1)\neq c(v_2)\}. 
\end{equation}
Let $G_1$ and $G_2$ be graphs. 
If $\#\underline{\chi}(G_1, X)=\#\underline{\chi}(G_2, X)$, 
then there exists a bijection $\varphi_X$ 
\[
\varphi_X:
\underline{\chi}(G_1, X)
\stackrel{\simeq}{\longrightarrow}
\underline{\chi}(G_2, X). 
\]
However, if we consider $\underline{\chi}(G_1, X)$ as a functor with respect to $X$, we may impose functoriality. The natural bijection should be formulated as an isomorphism 
\begin{equation}
\label{eq:isomfct}
\varphi: 
\underline{\chi}(G_1, \bullet)
\stackrel{\simeq}{\longrightarrow}
\underline{\chi}(G_2, \bullet) 
\end{equation}
of functors (chromatic functors). If the chromatic functors are isomorphic in the sense of (\ref{eq:isomfct}), then by substituting $X=[n]$, we can conclude that the graphs are chromatically equivalent. The main result asserts that the converse also holds. 

\begin{theorem}
(Theorem \ref{thm:main}) \\
Let $G_1$ and $G_2$ be finite graphs. Then they are chromatically equivalent if and only if associated chromatic functors are isomorphic as functors $(\ref{eq:isomfct})$. 
\end{theorem}

The advantage of chromatic functors over polynomials is that we can define chromatic functors also for graphs with infinite vertices. We say graphs $G_1$ and $G_2$, which are possibly infinite graphs, are \emph{chromatically equivalent} if associated chromatic functors are isomorphic in the sense of (\ref{eq:isomfct}). We will show several results on chromatic functors for infinite graphs in \S \ref{sec:infinitegraphs}. The first result is the following Cantor-Bernstein-Schr\"oder property. 

\begin{theorem}
(Theorem \ref{thm:cbs}) Let $G_1$ and $G_2$ be graphs. Suppose there exist surjective graph homomorphisms $G_1\twoheadrightarrow G_2$ and 
$G_2\twoheadrightarrow G_1$. Then $G_1$ and $G_2$ are chromatically equivalent.
\end{theorem}
As a corollary to the above theorem, we can show that connected unbounded countable trees determine isomorphic chromatic functors, which can be considered as a generalization of the result on the chromatic polynomial of connected finite trees. 

In \S \ref{sec:finchrnum}, we will characterize isomorphism classes of chromatic functors for infinite graphs with finite chromatic numbers. Roughly speaking, the isomorphism class of the functor $\underline{\chi}(G,\bullet)$ is characterized by the chromatic number $\chi(G)$ and cardinality $\#\underline{\chi}(G,[\chi(G)])$ if $G$ is a countable infinite graph with $\chi(G)<\infty$ (Theorem \ref{thm:finchrnum}). 

If $\chi(G)=\infty$, the chromatic functors detect more subtle information. In \S \ref{sec:example}, we exhibit two graphs $G_1$ and $G_2$ which satisfy 
\begin{itemize}
\item 
$\chi(G_1)=\chi(G_2)$, 
\item 
$
\#\underline{\chi}(G_1, [\chi(G_1)])=
\#\underline{\chi}(G_2, [\chi(G_2)])$, however, 
\item 
$G_1$ and $G_2$ are not chromatically equivalent. 
\end{itemize}

\section{Chromatic functors}

\label{sec:chromatic}

\subsection{Definitions and the main result}

Let $G_1=(V_1, E_1)$ and $G_2=(V_2, E_2)$ be simple graphs. A homomorphism $\varphi:G_1\longrightarrow G_2$ is a map $\varphi:V_1\longrightarrow V_2$ such that $\{v_1, v_2\}\in E_1\Longrightarrow \{\varphi(v_1), \varphi(v_2)\}\in E_2$. A homomorphism $\varphi:G_1\longrightarrow G_2$ is said to be surjective if the map of vertices $\varphi:V_1\longrightarrow V_2$ is surjective. We denote the category of graphs and homomorphisms, or surjective homomorphisms, by $\Graph$, or $\Graph^{\surj}$. 

Let $G=(V, E)$ be a graph and $X$ a set. Define the set $\underline{\chi}(G, X)\subset X^V$ by 
\begin{equation}
\underline{\chi}(G, X):=
\{c:V\longrightarrow X \mid 
\{v_1, v_2\}\in E\Longrightarrow c(v_1)\neq c(v_2)\}. 
\end{equation}
This defines a bi-functor 
\begin{equation}
\underline{\chi}:
\Graph\times\Set^{\inj}\longrightarrow\Set, \ 
(G, X)\longmapsto 
\underline{\chi}(G, X), 
\end{equation}
where $\Set$ (resp. $\Set^{\inj}$) is the category of sets and maps (resp. injections). Note that $\underline{\chi}(G, X)$ is covariant with respect to $X$ and contravariant with respect to $G$. If we restrict the functor to $\Graph^{\surj}$, then we have a bi-functor 
\[
\underline{\chi}: \Graph^{\surj}\times\Set^{\inj}
\longrightarrow
\Set^{\inj}. 
\]
These functors can also be formulated in terms of graph homomorphisms (\cite{koz}). Let us denote the complete graph with the set of vertices $X$ by $K_X$. Then, there is a natural isomorphism 
\[
\underline{\chi}(G, X)\simeq
\Hom_{\Graph}(G, K_X). 
\]

Let $G$ be a simple graph. Then $G$ determines a functor 
\[
\kasen(G, \maru): \Set^{\inj}\longrightarrow\Set^{\inj}, 
\]
which we call \emph{the chromatic functor} associated to $G$. If we restrict the chromatic functor to the category $\Set^{\bij}$ of sets and bijections, we obtain 
\[
\kasen(G, \maru): \Set^{\bij}\longrightarrow\Set^{\bij}. 
\]
Hence, the cardinality $\#\kasen(G, X)$ depends only on the cardinality $\#X$. As mentioned in \S \ref{sec:intro}, for a finite graph $G$, there exists a polynomial $\chi(G, t)\in\bZ[t]$ such that 
\[
\#\kasen(G, [n])=\chi(G, n), 
\]
for $n>0$, which is called the chromatic polynomial of $G$. If $\kasen(G_1, \maru)\simeq\kasen(G_2, \maru)$ as functors $\Set^{\inj}\longrightarrow\Set^{\inj}$, then obviously, $\sharp\kasen(G_1, [n])=\sharp\kasen(G_2, [n])$. In particular, they have the same chromatic polynomial. The following result asserts that the converse is also true. 

\begin{theorem}
\label{thm:main}
Let $G_1$ and $G_2$ be finite graphs. Then the following are equivalent. 
\begin{itemize}
\item[(i)] 
$\chi(G_1, t)=\chi(G_2, t)$. 
\item[(ii)] 
$\kasen(G_1, \maru)\simeq\kasen(G_2, \maru)$ as 
functors $\Set^{\inj}\longrightarrow\Set^{\inj}$. 
\end{itemize}
\end{theorem}
Theorem \ref{thm:main} will be proved later in \S \ref{sec:stablepart}. 

\begin{definition}
Let $G_1$ and $G_2$ be (possibly infinite) graphs. $G_1$ and $G_2$ are said to be \emph{chromatically equivalent} if 
\[
\kasen(G_1, \maru)\simeq\kasen(G_2, \maru), 
\]
as functors $\Set^{\inj}\longrightarrow\Set^{\inj}$. 
\end{definition}
Note that for finite graphs, by Theorem \ref{thm:main}, this definition is equivalent to the classical definition of chromatic equivalence \cite{bir}. 

\begin{remark}
The functorial formulation in this paper is related to the notion of ``combinatorial species'' \cite{joy, bll}, which considers a combinatorial object as a functor $\Set^{\bij}\longrightarrow\Set^{\bij}$. However, we consider the category $\Set^{\inj}$ instead of $\Set^{\bij}$. This requires stronger functoriality than combinatorial species. 
\end{remark}

\subsection{Stable partitions of chromatic functors}

\label{sec:stablepart}

Let $G=(V, E)$ be a simple graph. 

\begin{definition}
$\Pi=\{\pi_\lambda\mid \lambda\in\Lambda\}$ is called a \emph{stable partition} of $G$ if 
\begin{itemize}
\item 
$\bigsqcup\Pi:=\bigsqcup_{\lambda\in\Lambda}\pi_\lambda=V$, and 
\item 
$\{v_1, v_2\}\in E, v_1\in\pi_{\lambda_1}, v_2\in\pi_{\lambda_2}
\Longrightarrow \pi_{\lambda_1}\neq\pi_{\lambda_2}$. 
(In other words, $v_1, v_2\in\pi_\lambda
\Longrightarrow\{v_1, v_2\}\notin E$.)
\end{itemize}
\end{definition}
We denote the set of all stable partitions by 
\[
\St(G)=\{\Pi\mid\mbox{ $\Pi$ is a stable partition of $G$}\}. 
\]

\begin{example}
Let $X$ be a set. There are two extremal partitions: the trivial partition $\{X\}$ and the finest partition $\widetilde{X}:=\{\{x\}\mid x\in X\}$ into single elements. The stable partition $\St(K_X)$ of the complete graph $K_X$ consists of a single partition 
\[
\St(K_X)=\{\widetilde{X}\}. 
\]
\end{example}

\begin{proposition}
\label{prop:stablepart}
Let $G=(V, E)$ be a (possibly infinite) graph. Then 
\begin{equation}
\label{eq:stpartfct}
\kasen(G, \maru)\simeq
\bigsqcup_{\Pi\in\St(G)}
\kasen(K_\Pi, \maru). 
\end{equation}
\end{proposition}

\begin{proof}
Let $X$ be a set and $c:V\longrightarrow X$ a regular $X$-coloring. The coloring $c$ determines a stable partition 
\[
\Pi_c:=\{c^{-1}(x)\mid x\in X, c^{-1}(x)\neq\emptyset\}, 
\]
of $G$. This partition determines a surjective graph homomorphism 
\begin{equation}
\label{eq:partitionhom}
\varphi_{\Pi_c}:G\longrightarrow K_{\Pi_c}, 
\end{equation}
which sends $v\in V$ to $c^{-1}(c(v))\in\Pi_c$. By construction, the coloring $c:G\longrightarrow K_X$ factors through (\ref{eq:partitionhom}). Therefore, we have a map 
\[
\kasen(G, X)\longrightarrow
\bigsqcup_{\Pi\in\St(G)}
\kasen(K_{\Pi}, X). 
\]
Conversely, for any partition $\Pi\in\St(G)$, there is a surjective graph homomorphism $G\longrightarrow K_{\Pi}$. Because $\kasen$ is contravariant with respect to the graphs, we have $\kasen(G, X)\longleftarrow\kasen(K_{\Pi}, X)$. This induces 
\[
\kasen(G, X)\longleftarrow
\bigsqcup_{\Pi\in\St(G)}
\kasen(K_{\Pi}, X). 
\]
It is apparent that these maps are inverse to each other. 
\end{proof}

Now, let us prove (i) $\Longrightarrow$ (ii) of Theorem \ref{thm:main}. Let $G=(V, E)$ be a finite graph. Then there are only finitely many stable partitions. We decompose $\St(G)$ as 
\[
\St(G)=\bigsqcup_{k\geq 1}\St_k(G), 
\]
where, 
\[
\St_k(G)=\{\Pi\in\St(G)\mid \#\Pi=k\}. 
\]
Using Proposition \ref{prop:stablepart}, we have 
\begin{equation}
\label{eq:findecomp}
\kasen(G, \maru)\simeq
\bigsqcup_{k\geq 1}
\kasen(K_k, \maru)^{\sqcup\#\St_k(G)}. 
\end{equation}
This implies 
\begin{equation}
\label{eq:chrpolydec}
\chi(G, t)=\sum_{k\geq 1}\#\St_k(G)\cdot\chi(K_k, t). 
\end{equation}
Recall that $\chi(K_k, t)=t(t-1)\dots(t-k+1)$. Hence, the polynomials $\chi(K_k, t), (k\geq 1)$ are linearly independent. Therefore, from (\ref{eq:chrpolydec}), the number $\#\St_k(G)$ is uniquely determined by the chromatic polynomial $\chi(G, t)$. 

Now assume that two graphs $G_1$ and $G_2$ satisfy $\chi(G_1, t)=\chi(G_2, t)$. It follows that \#$\St_k(G_1)=\#\St_k(G_2)$ for any $k\geq 1$. Then using (\ref{eq:findecomp}), we may conclude that the associated chromatic functors are isomorphic. This completes the proof of Theorem \ref{thm:main}. 

\subsection{Uniqueness of stable partitions}

\label{sec:uniqueness}

We prove the uniqueness of the stable partition. The following is essentially the Yoneda lemma. 

\begin{lemma}
\label{lem:yoneda}
Let $X$ and $Y$ be sets. Suppose that 
\[
\widetilde{\alpha}:
\kasen(K_X, \maru)
\stackrel{\simeq}{\longrightarrow}
\kasen(K_Y, \maru)
\]
is an isomorphism of functors. Then $\widetilde{\alpha}$ is induced from a bijection $\beta:Y\stackrel{\simeq}{\longrightarrow}X$, that is, $\widetilde{\alpha}=\beta^*$. 
\end{lemma}

\begin{proof}
We denote the image of $\id_X\in\kasen(K_X, X)$ by the map $\widetilde{\alpha}_X:\kasen(K_X, X)\stackrel{\simeq}{\longrightarrow}\kasen(K_Y, X)$ by $\beta:=\widetilde{\alpha}_X(\id_X)$. Similarly, set $\gamma:=\widetilde{\alpha}_Y^{-1}(\id_Y)\in\kasen(K_X, Y)$. Let $c\kasen(K_X, Z)$ be an arbitrary coloring. By definition, $c:X\hookrightarrow Z$ is an injection. By the commutativity of the following diagram, 
\[
\begin{CD}
\kasen(K_X, X)@>\widetilde{\alpha}_X>> \kasen(K_Y, X)\\
@Vc_*VV @VVc_*V\\
\kasen(K_X, Z)@>\widetilde{\alpha}_Z>> \kasen(K_Y, Z), 
\end{CD}
\]
we have 
\[
\widetilde{\alpha}_Z(c)=
\widetilde{\alpha}_Z(c_*(\id_X))=
c_*(\widetilde{\alpha}_X(\id_X))=
c_*(\beta)=c\circ\beta=\beta^*(c). 
\]
Similarly, we have $\widetilde{\alpha}^{-1}=\gamma^*$. Using the same diagram with $Z=Y$ and $c=\gamma$, we have $\gamma\circ\beta=\id_Y$. Similarly, $\beta\circ\gamma=\id_X$. Thus, $\beta$ and $\gamma$ are bijective. 
\end{proof}
Let $X=\{x\mid x\in X\}$ and $Y=\{y\mid y\in Y\}$ be partitions of $\bigsqcup X:=\bigsqcup_{x\in X}x$ and $\bigsqcup Y:=\bigsqcup_{y\in Y}y$ respectively. Then $X$ determines a functor 
\[
F_X:=\bigsqcup_{x\in X}\kasen(K_x, \bullet): 
\Set^{\inj}\longrightarrow\Set^{\inj}. 
\]
Suppose there exists a bijection $\alpha:\bigsqcup X\stackrel{\simeq}{\longrightarrow}\bigsqcup Y$ that preserves partitions, that is, the restriction $\alpha|_x$ induces a bijection $\alpha|_x:x\stackrel{\simeq}{\longrightarrow}y$ for some $y\in Y$. Then clearly $\alpha$ induces an isomorphism of functors 
\[
\widetilde{\alpha}:F_X\stackrel{\simeq}{\longrightarrow}
F_Y. 
\]
The next result asserts that the converse also holds. 

\begin{proposition}
\label{prop:uniqueness}
Let $\widetilde{\varphi}: F_X\stackrel{\simeq}{\longrightarrow}F_Y$ be an isomorphism of functors. Then there exists a bijection $\varphi:\bigsqcup X\stackrel{\simeq}{\longrightarrow}\bigsqcup Y$ that preserves partitions. 
\end{proposition}

\begin{proof}
Let $x_0\in X$. Consider the identity map $\id_{x_0}:x_0\longrightarrow x_0$ as a coloring $\id_{x_0}\in\kasen(K_{x_0}, x_0)$. Then there exists a unique $y_0\in Y$ such that $\widetilde{\varphi}(\id_{x_0})\in\kasen(K_{y_0}, x_0)$. 

Let $c\in\kasen(K_{x_0}, Z)$ be an arbitrary coloring of $K_{x_0}$. Since $c:x_0\hookrightarrow Z$ is an injection, we may consider the following commutative diagram. 
\[
\begin{CD}
\kasen(K_{x_0}, x_0)@>\widetilde{\varphi}_{x_0}>> 
\bigsqcup\limits_{y\in Y}\kasen(K_y, x_0)\\
@Vc_*VV @VVc_*V\\
\kasen(K_{x_0}, Z)@>\widetilde{\varphi}_{Z}>> 
\bigsqcup\limits_{y\in Y}\kasen(K_y, Z)
\end{CD}
\]
By the commutativity of the diagram, we have $\widetilde{\varphi}_Z(c)=c_*(\widetilde{\varphi}(\id_{x_0}))$, which is contained in $\kasen(K_{y_0}, Z)$. Hence, $\widetilde{\varphi}$ induces a homomorphism $\widetilde{\varphi}: \kasen(K_{x_0}, \bullet)\longrightarrow \kasen(K_{y_0}, \bullet)$. It is an isomorphism, hence by Lemma \ref{lem:yoneda}, induced from a bijection $x_0\stackrel{\simeq}{\longrightarrow} y_0$. This enables us to conclude that the isomorphism $\widetilde{\varphi}$ is induced from the isomorphism of partitions $X\stackrel{\simeq}{\longrightarrow} Y$. 
\end{proof}

\section{Chromatic theory for infinite graphs}

\label{sec:infinitegraphs}

\subsection{Relative Cantor-Bernstein-Schr\"oder theorem}

\label{sec:relCBS}

In this section, we prove the following relative version of the Cantor-Bernstein-Schr\"oder (CBS) theorem. 

\begin{theorem}
\label{thm:relCBS}
Let 
$f:X_1\longrightarrow X_2$ and $g:Y_1\longrightarrow Y_2$ be maps of sets. Let $i_1, i_2, j_1$, and $j_2$ be injections that make the following diagram commutative. 
\begin{equation}
\label{eq:commdiags}
\begin{CD}
X_1 @>i_1>> Y_1\\
@VfVV @VVgV\\
X_2 @>i_2>> Y_2
\end{CD}
\hspace{2cm}
\begin{CD}
X_1 @<j_1<< Y_1\\
@VfVV @VVgV\\
X_2 @<j_2<< Y_2
\end{CD}
\end{equation}
Assume that for any subset $Z\subset X_2$ and $W\subset Y_2$, it holds that 
\begin{equation}
\label{eq:pullback}
g^{-1}(i_2(Z))=i_1(f^{-1}(Z)), \ \ 
f^{-1}(j_2(W))=j_1(g^{-1}(W)). 
\end{equation}
Then there exist bijections $r_\alpha:X_\alpha\stackrel{\simeq}{\longrightarrow}Y_\alpha$ ($\alpha=1, 2$) such that the following diagram is commutative. 
\begin{equation}
\label{eq:bijcomm}
\begin{CD}
X_1 @>r_1>\simeq> Y_1\\
@VfVV @VVgV\\
X_2 @>r_2>\simeq> Y_2
\end{CD}
\end{equation}
\end{theorem}

\begin{remark}
\label{rem:fiber}
The conditions (\ref{eq:pullback}) are satisfied if the diagrams (\ref{eq:commdiags}) are fiber products. 
\end{remark}

\begin{proof}
Define the subset $C_\alpha^n$ ($\alpha=1,2$ and $n\geq 0$) by 
\[
\left\{
\begin{array}{l}
C_\alpha^0:=X_\alpha\setminus j_\alpha(Y_\alpha), \\
\\
C_\alpha^{n+1}:=j_\alpha(i_\alpha(C_\alpha^n)), 
\end{array}
\right.
\]
and set $C_\alpha:=\bigcup_{n=0}^\infty C_\alpha^n$. Let us define the map $r_\alpha:X_\alpha\longrightarrow Y_\alpha$ ($\alpha=1, 2$) as follows. 
\[
r_\alpha(x)=
\left\{
\begin{array}{ll}
i_\alpha(x), &\mbox{ if }x\in C_\alpha, \\
&\\
j_\alpha^{-1}(x), &\mbox{ if }x\notin C_\alpha. 
\end{array}
\right.
\]
It is well known that $r_\alpha:X_\alpha\longrightarrow Y_\alpha$ is bijective. It remains for us to prove commutativity 
\begin{equation}
\label{eq:rfgr}
r_2\circ f=g\circ r_1. 
\end{equation}
First, let us prove 
\[
f^{-1}(C_2^n)=C_1^n
\]
by induction on $n$. In the case $n=0$, since $X_2=C_2^0\sqcup j_2(Y_2)$, we have 
\[
\begin{split}
X_1=f^{-1}(X_2)&=
f^{-1}(C_2^0)\sqcup f^{-1}(j_2(Y_2))\\
&=
f^{-1}(C_2^0)\sqcup j_1(g^{-1}(Y_2))\\
&=
f^{-1}(C_2^0)\sqcup j_1(Y_1). 
\end{split}
\]
Hence, $f^{-1}(C_2^0)=C_1^0$. Furthermore, 
\[
\begin{split}
f^{-1}(C_2^{n+1})
&=
f^{-1}(j_2(i_2(C_2^{n})))\\
&=
j_1(g^{-1}(i_2(C_2^{n})))\\
&=
j_1(i_1(f^{-1}(C_2^{n})))\\
&=
j_1(i_1(C_1^{n}))\\
&=
C_1^{n+1}. 
\end{split}
\]
Thus, we obtain $f^{-1}(C_2)=C_1$. It follows the commutativity relation (\ref{eq:rfgr}) from that of (\ref{eq:commdiags}). 
\end{proof}

\subsection{CBS property for chromatic functors}

\label{sec:cbs}

In this section, we prove that chromatic functors satisfy the CBS property. 

\begin{theorem}
\label{thm:cbs}
Let $G_1$ and $G_2$ be (possibly infinite) graphs. Suppose that there exist surjective graph homomorphisms $
\varphi:G_1\twoheadrightarrow G_2, \mbox{ and }
\psi:G_2\twoheadrightarrow G_1$. 
Then $G_1$ and $G_2$ are chromatically equivalent. 
\end{theorem}

\begin{proof}
Let $f:X\hookrightarrow Y$ be an injection of sets. Then we have the following commutative diagrams. 
\begin{equation}
\label{eq:chrcommdiags}
\begin{CD}
\kasen(G_1, X) @>\psi^*>> \kasen(G_2, X)\\
@Vf_*VV @VVf_*V\\
\kasen(G_1, Y) @>\psi^*>> \kasen(G_2, Y)
\end{CD}
\hspace{1cm}
\begin{CD}
\kasen(G_1, X) @<\varphi^*<< \kasen(G_2, X)\\
@Vf_*VV @VVf_*V\\
\kasen(G_1, Y) @<\varphi^*<< \kasen(G_2, Y)
\end{CD}
\end{equation}
Note that all maps in (\ref{eq:chrcommdiags}) are injective. It is easy to see that both diagrams are fiber products. Applying Theorem \ref{thm:relCBS} (see also Remark \ref{rem:fiber}), there exist bijections $r_X$ and $r_Y$ such that the following diagram is commutative. 
\begin{equation}
\label{eq:chrcommbij}
\begin{CD}
\kasen(G_1, X) @>r_X>\simeq> \kasen(G_2, X)\\
@Vf_*VV @VVf_*V\\
\kasen(G_1, Y) @>r_Y>\simeq> \kasen(G_2, Y)
\end{CD}
\end{equation}
The system of bijections $\{r_X\}$ determines an isomorphism of functors 
\[
r_{\maru}:\kasen(G_1, \maru)\stackrel{\simeq}{\longrightarrow}
\kasen(G_2, \maru). 
\]
This completes the proof. 
\end{proof}

\begin{remark}
\label{rem:hasebe1}
We can also prove the previous result using stable partitions. Namely, a graph surjective homomorphism $G_1\twoheadrightarrow G_2$ induces an injection $\St_\alpha(G_1)\hookleftarrow\St_\alpha(G_2)$ for any cardinality $\alpha$. Similarly, $G_1\twoheadleftarrow G_2$ induces an injection $\St_\alpha(G_1)\hookrightarrow\St_\alpha(G_2)$. Then the CBS Theorem concludes that there exists a bijection $\St_\alpha(G_1)\stackrel{\simeq}{\longrightarrow}\St_\alpha(G_2)$ for any cardinality $\alpha$. Then it follows from Proposition \ref{prop:stablepart} that $G_1$ and $G_2$ are chromatically equivalent. 
\end{remark}

\begin{definition}
\label{def:naturaltree}
The \emph{natural tree} $T_{\bN}=(V, E)$ (Figure \ref{fig:naturaltree}) is defined by 
\begin{itemize}
\item $V=\bN=\{0, 1, 2, \dots, n, \dots\}$, and 
\item $E=\{\{n, n+1\}\mid n\in\bN\}$. 
\end{itemize}
\begin{figure}[htbp]
\centering

\begin{picture}(300,20)(0,0)
\thicklines

\put(0,0){\line(1,0){280}}
\multiput(0,0)(50,0){6}{\circle*{4}}
\multiput(290,0)(10,0){3}{\circle*{2}}
\put(-3,4){$0$}
\put(47,4){$1$}
\put(97,4){$2$}
\put(147,4){$3$}
\put(197,4){$4$}
\put(247,4){$5$}

\end{picture}

\caption{The natural tree $T_{\bN}$}
\label{fig:naturaltree}
\end{figure}
\end{definition}

\begin{proposition}
\label{prop:cover}
Let $G=(V, E)$ be a connected countable graph. Then there exists a surjective graph homomorphism 
\[
\varphi:T_{\bN}\twoheadrightarrow G. 
\]
\end{proposition}

\begin{proof}
Since $V$ is countable and $G$ is connected, there exists a sequence $\{v_{i_0}, v_{i_1}, \dots, \}$ of vertices such that $\{v_{i_s}, v_{i_{s+1}}\}\in E$ and $\{v_{i_s}\}_{s\in\bN}$ covers $V$. Then $\bN\ni s\longmapsto v_{i_s}\in V$ induces a surjective graph homomorphism $\varphi:T_{\bN}\twoheadrightarrow G$. 
\end{proof}

Let $G=(V, E)$ be a connected graph. Then the set $V$ of vertices admits the adjacency metric, denoted by $d_G(\maru, \maru)$. The diameter of $G$ is defined by 
\[
\sup_{x, y\in V}d_G(x, y). 
\]
The graph $G$ is said to be unbounded if the diameter is $\infty$. 

\begin{theorem}
\label{thm:unbdd}
Let $G=(V, E)$ be a connected countable unbounded graph that has no cycle of odd length (in other words, $\chi(G)=2$). Then $G$ is chromatically equivalent to the natural tree $T_\bN$. 
\end{theorem}

\begin{proof}
Let $p\in V$. Then by the assumption that there are no cycles of odd length, the map $d_G(p, \bullet): V\longrightarrow \bN$ induces a graph homomorphism 
$d_G(p, \bullet): G\longrightarrow T_\bN$. The unboundedness of $G$ implies that the homomorphism $d_G(p, \bullet): G\longrightarrow T_\bN$ is surjective. Proposition \ref{prop:cover} and CBS-type result (Theorem \ref{thm:cbs}) conclude that $G$ and $T_\bN$ are chromatically equivalent. 
\end{proof}

\begin{corollary}
\label{cor:unbddtree}
Connected countable unbounded trees are chromatically equivalent to $T_\bN$. 
\end{corollary}

\subsection{Infinite graphs with finite chromatic numbers}

\label{sec:finchrnum}

\begin{definition}
\label{def:naturalwheel}
The \emph{natural wheel} $W_{\bN}=(V, E)$ (Figure \ref{fig:naturalwheel}) is defined by 
\begin{itemize}
\item $V=\bN=\{0, 1, 2, \dots, n, \dots\}$, and 
\item $E=\{\{0, n\}\mid n>0\}$. 
\end{itemize}
\begin{figure}[htbp]
\centering

\begin{picture}(300,60)(0,-10)
\thicklines

\put(0,0){\line(1,0){50}}
\multiput(0,0)(50,0){6}{\circle*{4}}
\multiput(270,0)(10,0){3}{\circle*{2}}
\put(-3,-12){$0$}
\put(47,-12){$1$}
\put(97,-12){$2$}
\put(147,-12){$3$}
\put(197,-12){$4$}
\put(247,-12){$5$}

\qbezier(0,0)(50,20)(100,0)
\qbezier(0,0)(75,50)(150,0)
\qbezier(0,0)(100,85)(200,0)
\qbezier(0,0)(125,120)(250,0)

\end{picture}

\caption{The natural wheel $W_{\bN}$}
\label{fig:naturalwheel}
\end{figure}
\end{definition}
The graph $W_\bN$ is a countable connected tree with diameter $2$. Since $W_\bN$ has a finite diameter, we cannot use Corollary \ref{cor:unbddtree}. However, in this section, we will prove that $W_\bN$ is also chromatically equivalent to $T_\bN$.  

Let $G=(V, E)$ be a countable infinite graph (we do not make any assumptions about diameter and connectivity). We also assume that the chromatic number $\chi(G)$ is finite. For instance, the graph $W_n$ satisfies these conditions ($\chi(W_n)=2$). Then we prove that the isomorphism class of the chromatic functor $\kasen(G, \maru)$ is determined by the cardinality $\#\kasen(G, [\chi(G)])$. 

\begin{theorem}
\label{thm:finchrnum}
Let $G_1=(V_1, E_1)$ and $G_2=(V_2, E_2)$ be countable infinite graphs with the same finite chromatic number $\chi(G_1)=\chi(G_2)=n$. Suppose that 
\[
\#\kasen(G_1, [n])=
\#\kasen(G_2, [n]), 
\]
(possibly infinite cardinality). Then $G_1$ and $G_2$ are chromatically equivalent. 
\end{theorem}

\begin{proof}
We shall use stable partitions of chromatic functors (Proposition \ref{prop:stablepart}). It is sufficient to show that 
\begin{equation}
\label{eq:coincidence}
\#\St_k(G_1)=
\#\St_k(G_2), 
\end{equation}
for all $k\leq\aleph_0$. Since 
\[
\#\kasen(G, [k])=\sum_{i=1}^k 
\begin{pmatrix}
k\\i
\end{pmatrix}
\cdot i!\cdot\#\St_i(G), 
\]
$\#\St_k(G_1)=\#\St_k(G_2)$ holds for $k\leq n$. Let $c\in\kasen(G_1, [n])$ be an element. Since $\#V_1=\aleph_0$, there exists $i_0\in [n]$ such that $\#c^{-1}(i_0)=\aleph_0$. For any subset $X\subset c^{-1}(i_0)$, let us define the map $\widetilde{c}:V\longrightarrow [n+1]$ by 
\[
\widetilde{c}(v)=
\left\{
\begin{array}{ll}
c(v)&\mbox{ if }c(v)\neq i_0\\
c(i_0)&\mbox{ if }c(v)= i_0 \mbox{ and }v\in X\\
n+1&\mbox{ if }c(v)= i_0 \mbox{ and }v\notin X. 
\end{array}
\right.
\]
Then $\widetilde{c}\in\kasen(G_1, [n+1])$. The map $\widetilde{c}$ changes according to the choice of the subset $X\subset c^{-1}(i_0)$. Hence, we have $\#\St_{n+1}(G_1)=2^{\aleph_0}$. Similarly, we have 
\[
\#\St_k(G_1)=\#\St_k(G_2)=2^{\aleph_0}
\]
for $n<k\leq\aleph_0$. This completes the proof. 
\end{proof}

\begin{remark}
\label{rem:hasebe2}
We can also prove Theorem \ref{thm:unbdd} using Theorem \ref{thm:finchrnum}. 
\end{remark}

\begin{example}
\label{ex:finchrnum}
Let $G_1, G_2$, and $G_3$ be graphs of Figure \ref{fig:G3}. Then they have the same chromatic number, $\chi(G_1)=\chi(G_2)=\chi(G_3)=3$. Note that $\#\kasen(G_1, [3])=2^{\aleph_0}$ and $\#\kasen(G_2, [3])=\#\kasen(G_3, [3])=6$. Hence, by Theorem \ref{thm:finchrnum}, $G_2$ and $G_3$ are chromatically equivalent, however, $G_1$ is not.

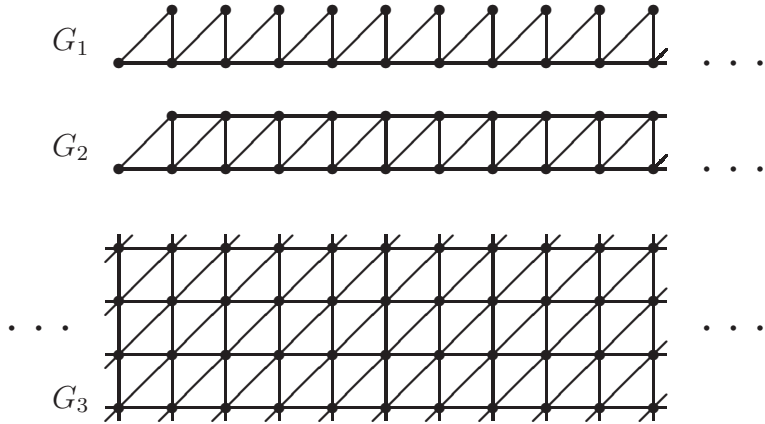
\begin{figure}[htbp]
\centering

\begin{picture}(200,150)(0,0)
\thicklines


\put(-25,135){$G_1$}
\put(0,130){\line(1,0){205}}
\multiput(220,130)(10,0){3}{\circle*{2}}
\multiput(0,130)(20,0){11}{\circle*{4}}
\multiput(20,150)(20,0){10}{\circle*{4}}
\multiput(0,130)(20,0){10}{\line(1,1){20}}
\multiput(20,130)(20,0){10}{\line(0,1){20}}
\qbezier(200,130)(202.5,132.5)(205,135)


\put(-25,95){$G_2$}
\put(0,90){\line(1,0){205}}
\multiput(220,90)(10,0){3}{\circle*{2}}
\multiput(0,90)(20,0){11}{\circle*{4}}
\multiput(20,110)(20,0){10}{\circle*{4}}
\multiput(0,90)(20,0){10}{\line(1,1){20}}
\multiput(20,90)(20,0){10}{\line(0,1){20}}
\qbezier(200,90)(202.5,92.5)(205,95)

\put(20,110){\line(1,0){185}}


\multiput(220,30)(10,0){3}{\circle*{2}}
\multiput(-20,30)(-10,0){3}{\circle*{2}}

\put(-25,0){$G_3$}
\multiput(0,0)(20,0){11}{\circle*{4}}
\multiput(0,20)(20,0){11}{\circle*{4}}
\multiput(0,40)(20,0){11}{\circle*{4}}
\multiput(0,60)(20,0){11}{\circle*{4}}

\multiput(0,-5)(20,0){11}{\line(0,1){70}}
\multiput(-5,0)(0,20){4}{\line(1,0){210}}

\put(-5,55){\line(1,1){10}}
\put(-5,35){\line(1,1){30}}
\put(-5,15){\line(1,1){50}}
\multiput(-5,-5)(20,0){8}{\line(1,1){70}}
\put(155,-5){\line(1,1){50}}
\put(175,-5){\line(1,1){30}}
\put(195,-5){\line(1,1){10}}

\end{picture}

\caption{$G_1, G_2$ and $G_3$}
\label{fig:G3}
\end{figure}

\end{example}

\subsection{Examples of non chromatically equivalent graphs}

\label{sec:example}

As presented in the previous section, the isomorphism class of the functor $\kasen(G, \maru)$ is determined by the chromatic number $\chi(G)$ and cardinality $\#\kasen(G, [\chi(G)])$, whenever $\chi(G)<\infty$. However, this is not the case for graphs with infinite chromatic number $\chi(G)$, two of which are presented here. 

\begin{example}
Let $X$ be an infinite set. Recall that $K_X=(V, E)$ is the complete graph with vertices $X$. Namely, $V=X$ and $E=\{e\in 2^X\mid \#e=2\}$. Let $e_0=\{x_0, x_1\}\in E$ be an edge, and define $E':=E\setminus\{e_0\}$. Consider the graph $K_X'=(V, E')$ obtained from $K_X$ by deleting an edge $e_0$. Then $K_X$ and $K_X'$ satisfy the following. 
\begin{itemize}
\item[(i)] $\chi(K_X)=\chi(K_X')=\#X$. 
\item[(ii)] $\#\kasen(K_X, X)=\#\kasen(K_X', X)=2^{\#X}$. 
\item[(iii)] $K_X$ and $K_X'$ are not chromatically equivalent. 
\end{itemize}
It is easily seen that (i) and (ii) hold. Let us prove (iii). First, consider the stable partition of $K_X'$. $\St(K_X')$ consists of two partitions, $\Pi_1:=\widetilde{X}:=\{\{x\}\mid x\in X\}$ and 
\[
\Pi_2:=\{\{x\}\mid x\in X, x\neq x_0, x_1\}\cup\{\{x_0, x_1\}\}. 
\]
Both have the same cardinality $\#\Pi_1=\#\Pi_2=\#X$. Therefore, the chromatic functor is isomorphic to 
\[
\kasen(K_X', \maru)\simeq 
\kasen(K_X, \maru)\sqcup 
\kasen(K_X, \maru). 
\]
However, by the uniqueness of stable partitions (Proposition \ref{prop:uniqueness}), it cannot be isomorphic to $\kasen(K_X, \maru)$. 
\end{example}

\medskip

\noindent
{\bf Acknowledgements.} The main part of this work was conducted during the author's stay at Hiroshima University in June 2015. The author thanks Professors Ichiro Shimada, Shun-ichi Kimura, Makoto Matsumoto, Nobuyoshi Takahashi, Akira Ishii, and Yuya Koda for many inspiring comments. In particular, the ideas of considering $\Set^{\inj}$ instead of $\Set^{\bij}$ and the distinction between bounded and unbounded trees were realized during our discussions. The author is also grateful to Professor Takahiro Hasebe for many suggestions (\S \ref{sec:uniqueness}, Remarks \ref{rem:hasebe1} and \ref{rem:hasebe2}). This work was partially supported by a Grant-in-Aid for Scientific Research (C) 25400060, JSPS.

\end{document}